\documentclass[10pt]{article}
\usepackage{a4wide}
\usepackage{amssymb}
\usepackage{amsfonts}
\usepackage{amsmath}

\newtheorem{theorem}{Theorem}

\newtheorem{corollary}[theorem]{Corollary}

\newtheorem{definition}[theorem]{Definition}

\newtheorem{proposition}[theorem]{Proposition}
\newtheorem{remark}[theorem]{Remark}

\DeclareMathOperator{\var}{Var}

\newcommand{\EE}{\mathbb{E}}
\newcommand{\XX}{\mathbb{X}}
\newcommand{\NN}{\mathbb{N}}
\newcommand{\PP}{\mathbb{P}}
\newcommand{\RR}{\mathbb{R}}

\newcommand{\cova}{{\rm Cov}}

\newenvironment{proof}[1][Proof]{\noindent\textbf{#1.} }{\ \rule{0.5em}{0.5em}}

\begin{document}

\title{An extremal property of the normal distribution, with a discrete analog%
}
\author{Erwan Hillion, Oliver Johnson, Adrien Saumard}
\maketitle

\begin{abstract}
We prove, using the Brascamp-Lieb inequality, that the Gaussian measure is the
only strong log-concave measure having a  strong log-concavity parameter equal
to its covariance matrix. We also give a similar characterization of the Poisson measure in the discrete case, using ``Chebyshev's other inequality''. We  briefly discuss how these results relate to Stein and Stein--Chen methods for Gaussian and Poisson approximation, and to the Bakry-\'{E}mery calculus.
\end{abstract}

\section{Introduction and definitions}

In this paper we consider probability densities on $(\RR^d,\mathcal{B}(\RR^d))$ which are strongly log-concave. Basic properties of log-concave and strong log-concave densities are given in the survey \cite{SauWel2014}. 

\begin{definition} \label{def:main}
Let $f : (\RR^d,\mathcal{B}(\RR^d)) \rightarrow \RR_+$ be a density function that is not supported  on any subspace of dimension $d-1$. We consider the potential function $\varphi : \RR^d \rightarrow ]-\infty,+\infty]$ defined by $\varphi = -\log(f)$. The density $f$ is said to be:
\begin{enumerate}
\item Log-concave if $\varphi$ is convex.
\item \label{it:slc} Strongly log-concave if there exists a symmetric positive definite $d \times d$ matrix $\Sigma$ such that the ratio $g := f/\gamma_{\mu,\Sigma}$  is log-concave where $\gamma_{\mu ,\Sigma }$ is the density of the Gaussian measure of mean $\mu $ and covariance matrix $\Sigma $, denoted $N_{d}\left(\mu ,\Sigma \right) $. In this case we write $f \in SLC(\Sigma,d)$, and refer to $\Sigma$ as the strong log-concavity parameter of $f$.
\end{enumerate}
\end{definition}

For brevity, a $d$-dimensional random vector $X$ is said to belong to $SLC(\Sigma,d)$ if it admits a density $f$ with respect to the Lebesgue measure on $\RR^d$ such that $f \in SLC(\Sigma,d)$. 
Observe that in Definition \ref{def:main}.\ref{it:slc}, the choice of $\mu$ is irrelevant; if the ratio $f/\gamma_{\mu,\Sigma}$ is log-concave for some $\mu$, it is log-concave for all $\mu$ (since the second derivative of $\log \gamma_{\mu,\Sigma}$ does not depend on $\mu$). For simplicity authors often choose to take $\mu$ to be zero, or to equal the 
expectation of $f$.

We state the following two results without proof:

\begin{proposition}
A strongly log-concave measure is log-concave and when its potential is twice differentiable, belonging to  $SLC\left( \Sigma ,d\right) $ is equivalent to having
\begin{equation} \label{eq:slcdef}
\varphi^{\prime \prime }\left( x\right) \succeq \Sigma ^{-1}
 \mbox{\;\;\; for any $x\in \RR^{d},$}
\end{equation}
where $\varphi^{\prime \prime }$ is the Hessian matrix of $\varphi $ and the order relation is the natural (Loewner) partial order for semi-definite symmetric matrices. Using the fact that this order is reversed on taking inverses (see
 \cite[Proposition 8.6.6]{bernstein}), we can also write this in
the form
\begin{equation} \label{eq:slcdef2}
  \left( \varphi^{\prime \prime }\left( x\right) \right)^{-1}  \preceq \Sigma
\mbox{\;\;\; for any $x\in \RR^{d},$}
\end{equation}
\end{proposition}

In the one-dimensional case, for $\alpha>0$, we write $SLC(\alpha)=SLC(\alpha,1)$ and we have:

\begin{proposition}
A differentiable density function $f$ is in $SLC(\alpha)$ if and only if the function $\frac{f^{\prime}(x)}{f(x)}+\frac{x}{\alpha}$ is non-increasing in $x$.
\end{proposition}

Clearly, by definition any Gaussian $X$ belongs to $SLC(\Sigma,d)$ with strong log-concavity parameter $\Sigma$ equal to the covariance matrix $\cova(X,X)$. Indeed in general the strong log-concavity parameter $\Sigma$ is sometimes (erroneously)
called the covariance parameter. A natural question is therefore the following: if $X$ is a random vector belonging to $SLC(\Sigma,d)$, can we relate the strong log-concavity parameter $\Sigma$ to the covariance matrix $\cova(X,X)$? \\

In this note, we answer this question by proving the inequality $ \cova(X,X) \preceq \Sigma  $ (in Theorem \ref{thm:ctsbl}). Moreover, we deduce a characterization of the Gaussian; there is equality $ \Sigma= \cova(X,X) $  if and only if $X$ has Gaussian distribution with covariance $\Sigma$. In Section~\ref{sec:BL}, we prove this fact using the Brascamp--Lieb inequality. In Section~\ref{sec:FKG}, we prove a more general characterization of Gaussian distributions, but in the restricted framework of one-dimensional distributions. In Section~\ref{sec:Poisson}, we use similar methods to prove a characterization of Poisson distributions.

\section{The continuous case via the Brascamp-Lieb inequality} \label{sec:BL}

\begin{theorem} \label{thm:ctsbl}
\label{prop_extrem_Gauss}
Suppose that the random vector $X\in SLC\left(
\Sigma ,d\right) $ for some symmetric  positive definite matrix $\Sigma $.  
\begin{enumerate}  
\item[(a)]   Then 
\begin{equation} \label{eq:ctsbl} \cova(X,X) \preceq \Sigma. \end{equation}
\item[(b)]  If $\cova(X,X) = \Sigma$ then $X$ has a multivariate normal distribution with
covariance matrix $\Sigma $, that is $X\sim 
\mathcal{N}_{d}\left( \mu ,\Sigma \right) $ for some $\mu$.
\end{enumerate}
\end{theorem}

Let us recall the celebrated Brascamp-Lieb inequality \cite{Brascamp}, that can be thought of as a weighted Poincar\'{e} inequality, and which will be instrumental in our proof. If density $f$ is strictly log-concave, its potential $\varphi $ is twice continuously differentiable and $g \in L_{2}(f)$ is continuously differentiable, then for $X \sim f$
\begin{equation} \label{eq:brascamp}
\var (g(X))\leq \EE \left[ \nabla g(X)^{T}(\varphi ^{\prime \prime
}(X))^{-1}\nabla g(X)\right] \text{ }.
\end{equation}%
%where $\varphi ^{\prime \prime }$ stands for the Hessian of $\varphi $.

Theorem \ref{prop_extrem_Gauss} also builds upon the work of Chen and Lou \cite{Chen}
on characterization of the Gaussian distribution by the Poincar\'{e}
inequality. Indeed, Corollary 2.1\ in \cite{Chen} can be stated as
follows. Let $X=\left( X_{1},...,X_{d}\right) $ be a random vector such that 
$\var \left( X_{j}\right) =\sigma _{j}^{2}>0$ for any $j\in \left\{
1,...,d\right\} $. Define%
\begin{equation}
U\left( X,\Sigma \right) =\sup_{g\in \mathcal{H}_{X}}\frac{%
\var
\left( g\left( X\right) \right) }{\EE \left[ \nabla ^{T}g\left(
X\right) \Sigma \nabla g\left( X\right) \right] }\text{ ,}  \label{def_U}
\end{equation}%
where $\Sigma $ is a $d\times d$ positive semidefinite matrix with $\sigma
_{1}^{2},...,\sigma _{d}^{2}$ as its diagonal elements and $\mathcal{H}%
_{X}=\left\{ g\in C^{1}\left( \RR^{d}\right) \cap L_{2}\left(
X\right) :\EE\left[ \nabla ^{T}g\left( X\right) \Sigma \nabla g\left(
X\right) \right] >0\right\} $. Clearly taking $g (x) = x_i$ for any $i$, we can deduce that $U\left( X,\Sigma \right) \geq 1$.
However \cite[Corollary 2.1]{Chen} shows that this is sharp, by proving that
$ U\left( X,\Sigma \right) =1$ if and only if $X$ has a multivariate normal
distribution with covariance matrix $\Sigma $.

\bigskip 

\begin{proof}[Proof of Theorem \protect\ref{prop_extrem_Gauss}]
Assume first that the potential $\varphi =-\log f$ of the density $f$ of $X$ is twice  continuously
differentiable. Then combining the Brascamp-Lieb inequality \eqref{eq:brascamp} with the assumption \eqref{eq:slcdef2}, for any
continuously differentiable 
 $g\in L_{2}\left( f \right) $ we have:
\begin{eqnarray}
\var (g(X)) & \leq & \EE\left[ \nabla g(X)^{T}(\varphi ^{\prime \prime
}(X))^{-1}\nabla g(X)\right]  \nonumber \\ 
& \leq & \EE\left[ \nabla ^{T}g\left(
X\right) \Sigma \nabla g\left( X\right) \right] \text{ .} \label{eq:keyineq}
\end{eqnarray}
We can deduce that Equation \eqref{eq:ctsbl} holds; for any vector $u \in \RR^d$ we can take the linear function $g(x) = \sum_{i=1}^d u_i x_i$ in
\eqref{eq:keyineq} to deduce that $u^T \cova(X,X) u \leq u^T \Sigma u$. Since this holds for any $u$, we deduce that
$\cova(X,X) \preceq \Sigma$ in the partial order sense as claimed in part (a) of the theorem.

In general, note that
approximation by convolution with Gaussian vectors allows us to reduce to the
case where $\varphi $ is twice continuously differentiable. In
particular, it regularizes the potential of any strongly log-concave
measure, while preserving strong-log-concavity (see \cite{SauWel2014},
Proposition 5.5), meaning that (a) holds for all SLC $f$.

To deduce the case of equality stated in (b), we can restate \eqref{eq:keyineq} to say that if $X\in SLC\left(
\Sigma ,d\right) $ then the quantity defined in (\ref{def_U}) 
satisfies $U\left( X,\Sigma \right) \leq 1$. But we already have $U\left( X,\Sigma \right) \geq 1$. Hence $%
U\left( X,\Sigma \right) =1$, which implies by  \cite[Corollary 2.1]{Chen},  that $X$ is a multivariate
normal distribution with covariance matrix $\Sigma $. 
\end{proof}

\medskip

Notice that a careful reading of the proof of Theorem \ref{prop_extrem_Gauss} shows that we can weaken the assumption in the case of equality. That is, following \cite[Corollary 2.1]{Chen}, it is sufficient that $X\in SLC\left(
\Sigma ,d\right) $ for some $\Sigma$ with diagonal elements $\Sigma _{jj}= \var(X_j)  =\sigma _{j}^{2}>0$ for each  $j\in \left\{
1,...,d\right\} $ to deduce that $X\sim 
\mathcal{N}_{d}\left( \mu ,\Sigma \right) $ for some $\mu$.

\section{A one-dimensional approach using ``Chebyshev's other inequality''} \label{sec:FKG}

In this paragraph we consider probability measures on a space $\XX$ which can
be either the real line $\RR$ (with the Borel $\sigma$-algebra), the
set of natural integers $\NN$ or the discrete interval $\{0,\ldots,N\}$%
. In each case, $\XX$ is a totally ordered set, on which the following
inequality   holds:

\begin{proposition}
\label{prop:FKG} Let $u,v: \XX \rightarrow \RR$ be two functions which
are either both non-decreasing or both non-increasing. Let $X$ be a $\XX$%
-valued random variable such that $\EE[u(X)^{2}]$ and $\EE%
[v(X)^{2}]$ are both finite. We then have 
\begin{equation}
\EE[u(X)v(X)]\geq \EE[u(X)]\EE[v(X)],
\end{equation}%
which can also be written as 
\begin{equation}
\cova(u(X),v(X))\geq 0.  \label{eq:FKG}
\end{equation}

If furthermore, we suppose that $u$ is non-decreasing, $v$ is strictly
increasing and that the covariance $\cova(u(X),v(X))$ is $0$, then $u$ is a constant function on the image of $X$.
\end{proposition}

Proposition~\ref{prop:FKG} is known as``Chebyshev's other inequality''  (see for example Kingman \cite[Eq. (1.7)]{kingman}), or as the
FKG inequality, due to a generalization of equation~\eqref{eq:FKG} to the
framework of finite distributive lattices, see \cite{FKG}. For the sake of
completeness, we give here a short proof:

\medskip

\begin{proof}[Proof of Proposition~\ref{prop:FKG}] We simply notice that: 
\begin{equation}
2 \cova(u(X),v(X))=\EE\left[ (u(X_{1})-u(X_{2}))(v(X_{1})-v(X_{2}))\right] ,
\end{equation}%
where $(X_{1},X_{2})$ are two independent copies of $X$. The monotonicity
assumption on $u$ and $v$ shows that $(u(X_{1})-u(X_{2}))(v(X_{1})-v(X_{2}))$
is  non- negative for all $X_1$ and $X_2$, which gives the inequality on the covariance.

If we have $
\cova (u(X),v(X))=0$ then $(u(X_{1})-u(X_{2}))(v(X_{1})-v(X_{2}))=0$ a.s. But the
assumption on $v$ implies that $u(X_{1})=u(X_{2})$ a.s. As $X_{1}$ and $%
X_{2} $ are independent, this means that $u$ is constant on the image of $X$.
\end{proof}

We use Proposition \ref{prop:FKG} to deduce the following result, which can be seen as a strengthening of 
Theorem \ref{prop_extrem_Gauss} in the one-dimensional case (see Corollary \ref{cor:1dbd}). It thus provides a link between the Brascamp--Lieb inequality
\cite{Brascamp} and Chebyshev's other inequality.

\begin{proposition} \label{prop:pseudoStein}
Let $X$ be a real-valued random variable
with mean $\mu$ and density $f$, where $f$ is in the class $%
SLC(\alpha )$ for some $\alpha >0$. Let $v\in \mathcal{C}^{1}(\RR,\RR)$ be
stricly increasing. Then 
\begin{equation}
\alpha \EE[v^{\prime }(X)]\geq \EE[ (X- \mu) v(X)].
\label{eq:pseudoStein}
\end{equation}%
Furthermore, if for one such function $v$, equality is attained in
inequality~\eqref{eq:pseudoStein}, then $X \sim \mathcal{N}(\mu ,\frac{1}{%
\alpha })$ for some $\mu \in \RR$.
\end{proposition}

\begin{proof} We set $u(x) := \frac{f^{\prime}(x)}{f(x)}+\frac{x-\mu}{\alpha}$, which has mean $\EE u(X) = 0$.
By the $SLC(\alpha)$ assumption, we know that $u : \RR \rightarrow 
\RR$ is non-increasing. By Proposition~\ref{prop:FKG}, we have $%
\EE[u(X)v(X)] = \cova(u(X),v(X)) \leq 0$.  But we have : 
\begin{eqnarray*}
\EE[u(X)v(X)] &= & \int_\RR \left(\frac{f^{\prime}(x)}{f(x)} + \frac{x - \mu
}{\alpha} \right) v(x) f(x) dx \\
&=& \int_\RR f^{\prime}(x) v(x) dx + \frac{1}{\alpha} \int_\RR (x - \mu) v(x) f(x) dx
\\
&=& \int_\RR f(x) \left( \frac{x - \mu}{\alpha} v(x)-v^{\prime}(x) \right) dx \\
&=& \frac{1}{\alpha }\EE[(X-\mu) v(X) - \alpha v^{\prime}(X)],
\end{eqnarray*}
from which we deduce the inequality we wanted.

If equality is attained for some strictly increasing function $v$, we deduce from the equality
case in Proposition~\ref{prop:FKG} that $u(X)$ is a constant random variable,
thus that $u$ is a constant function on the support of $X$.  However, the SLC assumption on $X$ means that the potential $\varphi$ is convex, which tells us that the support of $X$  (the values for which $\phi$ is finite) is an interval.

On this interval, we can consider the solutions $f$
of the ODE $\frac{f^{\prime}(x)}{f(x)}+\frac{x-\mu}{\alpha} = A$, which satisfy $%
f(x) = B e^{Ax-\frac{(x-\mu)^2}{2 \alpha}}$, for some constants $A,B \in 
\RR$. The constraint that $f$ is a probability density with mean $\mu$ implies that $A=0$ and $B=%
\sqrt{\frac{1}{2 \pi \alpha }}$, which means that $X \sim \mathcal{N}%
(\mu,\alpha)$. \end{proof}

\medskip

An immediate corollary of Proposition \ref{prop:pseudoStein}, which is the one-dimensional case of  Theorem \ref{prop_extrem_Gauss},  is obtained by considering the case where $v(x)=x-\mu$:

\begin{corollary} \label{cor:1dbd}
Let $X$ be a real-valued random variable with density $f$, where $f$ is in
the class $SLC(\alpha)$ for some $\alpha > 0$. Then $\var(X) \leq
\alpha$, with equality if and only if $X$ is Gaussian.
\end{corollary}
Indeed, writing $M_{r}(X) = \EE (X-\mu)^r$ for the centred moments of order $r$ and taking $v(x) = (x-\mu)^{2r-1}$, 
for any $X \in SLC(\alpha)$ we can deduce that  $M_{2r}(X) \leq \alpha (2r-1) M_{2r-2}(X)$, and hence by induction $M_{2r}(X) \leq (2r)!/r! (\alpha/2)^r$, so the values obtained by the Gaussian are extremal, as we might expect.

\begin{remark} \label{rem:stein}
Equation \eqref{eq:pseudoStein} can be viewed as a one-sided version of the Stein equation, used to establish a characterization of the Gaussian distribution when proving the Central Limit Theorem in Stein's Method \cite{Stein}. That is if, with $\alpha = \var(X)$, the equation \eqref{eq:pseudoStein} holds with equality for all $v$ then it is well-known that $X$ must be Gaussian. Here, Proposition \ref{prop:pseudoStein} allows us to reach the same conclusion if equality holds for a single $v$, under the additional SLC assumption.
\end{remark}

\section{A characterization of Poisson distributions.} \label{sec:Poisson}

The same strategy can be adapted to the discrete case, to study random variables supported on the natural numbers $\NN$, with suitable
definitions of derivative and of strong log-concavity:

\begin{definition} \mbox{ }
\begin{enumerate}
\item
The left-derivative $\nabla u$ of a function $u : \NN \rightarrow 
\RR$ is defined by $\nabla u(0) :=u(0)$, and by $\nabla u(n) :=
u(n)-u(n-1)$ for $n \geq 1$.
\item
The right-derivative $\nabla^* v$ of a function $v : \NN \rightarrow 
\RR$ is defined by $\nabla^* v(n) :=- v(n+1)-v(n)$ for $n \geq 0$.
\end{enumerate}
\end{definition}

The operators $\nabla $ and $\nabla ^{\ast }$ are dual up to a sign, in the
sense that a simple application of summation by parts gives
\begin{equation}
\sum_{n=0}^{\infty }\left( \nabla u(n)\right) v(n)= -\sum_{n=0}^{\infty
}u(n)\nabla ^{\ast }v(n),
\end{equation}%
for every function $u,v\in \mathcal{L}^{2}(\NN)$.

\begin{definition} Consider probability mass function $f : \NN \rightarrow \RR_+^*$ such that $\sum_{k =
0}^\infty f(k) = 1$ and $\alpha > 0$. We say that $f \in SLC(\alpha)$ if 
sequence $\left(\frac{\nabla f(k)}{f(k)} + \frac{k}{\alpha} \right)_{k \geq
0}$ is non-increasing in $k$.
\end{definition}

Direct calculations show that $f \in SLC(\alpha)$ if and only if $f(1) \leq
\alpha f(0)$ and 
\begin{equation} \label{eq:slcdisc}
\forall n \geq 0 \ , \  \frac{f(n+1)}{f(n+2)} - \frac{ f(n)}{f(n+1)} = \frac{f(n+1)^2-f(n) f(n+2)}{f(n+1) f(n+2)} \geq \frac{%
1}{\alpha}.
\end{equation}
We note that this condition was introduced as a special case of Assumption A in \cite{caputo}, and was studied further in
\cite{johnson38}.  In the case of  Poisson random variables with mean $\mu$ observe that the LHS of \eqref{eq:slcdisc} is constant and equal to $1/\mu$, so Poisson random variables are $SLC(\alpha)$ where strong log-concavity parameter $\alpha = \mu$. Again, we shall see that this property characterizes the Poisson family, using the following result.

\begin{proposition}
Let $X$ be a $\NN$-valued random variable with mean $\mu$ such that for every $n \geq 0$%
, $\PP(X=n)=f(n)$, where $f \in SLC(\alpha)$. For every strictly increasing $v : \NN
\rightarrow \RR$ the
\begin{equation}  \label{eq:pseudoSteinN}
\alpha \EE[ \nabla^* v(X)] \geq \EE[ (X-\mu) v(X)].
\end{equation}
Furthermore, if equality is attained in equation~\eqref{eq:pseudoSteinN} for some $v$,
then $X$ is Poisson with mean $\alpha$ (we write $X \sim \mathcal{P}(\alpha)$).
\end{proposition}
\begin{proof} We again apply Proposition~\ref{prop:FKG} with the
functions $u(k)=\frac{\nabla f(k)}{f(k)}+\frac{k - \mu}{\alpha }$ and $v(k)$,
yielding $\EE[u(X)v(X)]\leq 0$. But : 
\begin{eqnarray*}
\EE[u(X)v(X)] &=&\sum_{k=0}^{\infty }\left( \frac{\nabla f(k)}{f(k)}+%
\frac{k-\mu}{\alpha }\right) v(k)f(k) \\
&=&\sum_{k=0}^{\infty }\left( \nabla f(k)\right) v(k)+\frac{1}{\alpha }%
\sum_{k=0}^{\infty } (k - \mu) v(k)f(k) \\
&=&\frac{1}{\alpha }\sum_{k=0}^{\infty }f(k)\left(( k- \mu) v(k)-\alpha \nabla
^{\ast }v(k)\right) \\
&=&\frac{1}{\alpha }\EE[ (X-\mu) v(X)-\alpha (\nabla ^{\ast }v(X)].
\end{eqnarray*}%
If equality is attained for some strictly increasing $v$, we deduce that $u$
is a constant function, i.e. that there is some $\lambda \in \RR$
such that : 
\begin{equation}
\forall n\geq 0\ ,\ \frac{\nabla f(n)}{f(n)}+\frac{n}{\alpha }=\lambda .
\label{eq:EDOPoisson}
\end{equation}%
But equation~\eqref{eq:EDOPoisson} with $n=0$ implies that $\lambda =1$, and
thus for $n\geq 1$, equation~\eqref{eq:EDOPoisson} takes the simpler form $%
\frac{f(n-1)}{f(n)}=\frac{n}{\alpha }$, from which we deduce that $f(n)=%
\frac{f(0)\alpha ^{n}}{n!}$ for every $n\geq 0$. The condition $%
\sum_{n=0}^{\infty }f(n)=1$ gives $f(0)=e^{-\alpha }$, and we recognize the
Poisson distribution $X \sim \mathcal{P}(\alpha )$.
\end{proof}
\bigskip

Again, by taking $v(x) = x - \mu$ in  \eqref{eq:pseudoSteinN}, we can deduce that the $SLC(\alpha)$ condition can only hold if
$\alpha \geq \var(X)$, which we can view as a discrete counterpart of Theorem \ref{thm:ctsbl}. Note that \cite[Lemma 5.3]{johnson38} showed that the same condition implies that $\alpha \geq \mu$.

A counterpart of Remark \ref{rem:stein} holds on $\NN$, referring to the Stein--Chen method in Poisson approximation \cite{chen75}.
That is the Stein--Chen method is based on the fact that if  \eqref{eq:pseudoSteinN} holds with equality for $\alpha = \mu$ for 
every function $v$, then we can deduce that $X$ must be Poisson. Again, we are able to reach the same conclusion under the SLC
condition if equality is attained for a single function $v$.

\begin{remark} One further link between discrete and continuous settings is the following. It is well-known that strong log-concave densities satisfy the so-called Bakry-\'{E}mery condition, which is a natural setting under which functional inequalities (including Poincar\'{e} and log-Sobolev) can be proved, with lower bounds on $\varphi''$ of the form \eqref{eq:slcdef} guaranteeing bounds on the log-Sobolev constant -- see for example \cite{bakry2} for a review of this material.  It is striking that \cite{johnson38} proved similar results on $\NN$ with a similar role being played by the value of $\alpha$ arising in the discrete SLC condition \ref{eq:slcdisc}. The results of the current paper give further evidence of a natural link between these formulations. \end{remark}

We briefly remark that similar arguments can be used to characterize the binomial distribution among random variables with probability mass functions $f$ supported on discrete interval $\{0,\ldots,N\}$. That is, if we define derivative $\nabla_N h(0) = h(0)$ and
\begin{equation} \label{eq:nablan}
 \nabla_N h(n) := \frac{N-n}{N} h(n) - \frac{N-n+1}{N} h(n-1) \mbox{ \;\;\; for $1 \leq n \leq N$},\end{equation}
 and its conjugate to satisfy $\nabla_N^* h(N) = 0$ and
\begin{equation} \label{eq:nablanst}
 \nabla_N^* h(n) := \frac{N-n}{N} (h(n+1) - h(n)) = \frac{N-n}{N} \nabla^* h(n)  \mbox{ \;\;\; for $0 \leq n \leq N-1$},\end{equation}
we can define the set of $SLC_N(\alpha)$ random variables to be those for which $u(n) := \frac{\nabla_N f(n)}{f(n)} + \frac{n-\mu}{\alpha}$ is non-increasing in $n$. Observe that taking $f$ to be  Binomial$(N,p)$ random variables,  this property holds  with equality if $\alpha = N p = \mu$.

Again, using the same argument based on Chebyshev we can deduce that for  random variables $X \in SLC_N(\alpha)$ and  strictly increasing functions $v$, the expectation
\begin{equation} \label{eq:charbin}
\alpha \EE \left[  \nabla_N^* v(X) \right] \geq \EE \left[ (X- \mu) v(X) \right]. \end{equation}
Again taking $v(x) =  x- \mu$ we deduce that $\var(X) \leq \alpha(1 - \mu/N)$.  Note that  $\var(X) = N p(1-p) = \alpha(1-\mu/N)$  if $X$ is Binomial$(N,p)$.  Indeed as before, if equality in \eqref{eq:charbin} holds for some strictly increasing $v$, a similar argument based on  $u(n)$ being constant allows us to deduce that $f(n) = \binom{N}{n} q^n (1-q)^{N-n}$, where $q = \alpha/N$, so $f$ is Binomial with mean $\alpha$.


\begin{thebibliography}{BGL14}

\bibitem[BGL14]{bakry2}
D. Bakry, I. Gentil, and M.l Ledoux.
\newblock {\em Analysis and geometry of Markov diffusion operators}, volume 348
  of {\em Grundlehren der mathematischen Wissenschaften}.
\newblock Springer, 2014.

\bibitem[B09]{bernstein}
D.~S. Bernstein.
\newblock {\em Matrix mathematics. Theory, facts, and formulas. Second edition.}
\newblock  Princeton University Press, Princeton, NJ, 2009.

\bibitem[BL76]{Brascamp}
H.~J. Brascamp and  E.~H. Lieb.
\newblock On extensions of the Brunn--Minkowski and Pr\'{e}kopa--Leindler theorems, including inequalities for log concave functions, and with an application to the diffusion equation.
\newblock {\em J. Funct. Anal.}. 22: 366–389, 1976.

\bibitem[C09]{caputo}
P. Caputo, P. Dai~Pra, and G. Posta.
\newblock Convex entropy decay via the {B}ochner-{B}akry-{E}mery approach.
\newblock {\em Ann. Inst. Henri Poincar\'e Probab. Stat.}, 45(3):734--753,
  2009.

\bibitem[C75]{chen75}
L.~H.~Y. Chen.
\newblock Poisson approximation for dependent trials.
\newblock {\em Ann. Probab.}, 3:534--545, 1975.


\bibitem[CL87]{Chen} L.~H.~Y. Chen and J.~H. Lou.
\newblock Characterization of probability distributions by {P}oincar\'e-type
  inequalities.
\newblock {\em Ann. Inst. H. Poincar\'e Probab. Statist.}, 23(1):91--110, 1987.

\bibitem[FKG71]{FKG}
C.~M. Fortuin, P.~W. Kasteleyn and J. Ginibre.
\newblock Correlation inequalities on some partially ordered sets.
\newblock {\em Comm. Math. Phys.}, 22: 89–103, 1971. 

\bibitem[J17]{johnson38}
O.~T. Johnson.
\newblock A discrete log-{S}obolev inequality under a {Bakry-\'Emery} type
  condition.
\newblock {\em Annales de l'Institut Henri Poincar{\'{e}} B (Probability and
  Statistics)}, 53(4):1952--1970, 2017.

\bibitem[K78]{kingman}
J.~F.~C. Kingman.
\newblock Uses of exchangeability.
\newblock {\em Ann. Probability}, 6(2):183--197, 1978.

\bibitem[SW14]{SauWel2014}  
A. Saumard,  and J.~A. Wellner.
\newblock Log-concavity and strong log-concavity: {A} review.
\newblock {\em Stat. Surv.}, 8:45--114, 2014.

\bibitem[S71]{Stein}
C.~Stein.
\newblock A bound for the error in the normal approximation to the distribution of a sum of dependent random variables. 
\newblock {\em Proceedings of the Sixth Berkeley Symposium on Mathematical Statistics and Probability (Univ. California, Berkeley, Calif., 1970/1971), Vol. II: Probability theory}, pp. 583–602. Univ. California Press, Berkeley, Calif., 1972. 

\end{thebibliography}
\end{document}